\documentclass[12pt]{amsart}
\usepackage{amssymb,amsthm,amstext,amsmath,amscd,latexsym,amsfonts}
\def\FMN{{\mathcal F}_{M, N}} 
\def\grmod{\mathrm{mod}^{\Z}}

\def\grmodR{\mathrm{mod}^{\Z}(R)}
\def\grCMR{\mathrm{CM}^{\Z}(R)}
\def\grCMPSFR{\mathrm{CM}_{0}^{\Z}(R)}
\def\grHom{{}^{\ast}\mathrm{Hom}}
\def\grHomR{{}^{\ast}\mathrm{Hom}_R}

\def\grExt{{}^{\ast}\mathrm{Ext}}

\def\ulM{{\underline{M}}}
\def\ulN{{\underline{N}}}
\def\ulZ{{\underline{Z}}}
\def\SgrCM{{\underline{\mathrm{CM}}^{\Z}}}

\def\SgrCMR{{\underline{\mathrm{CM}}^{\Z}}(R)}

\def\homo{\leq _{hom}}
\def\dego{\leq _{deg}}
\def\exto{\leq _{ext}}

\def\Hom{\mathrm{Hom}}
\def\HomR{\mathrm{Hom}_R}
\def\End{\mathrm{End}}
\def\Ext{\mathrm{Ext}}

\def\Im{\mathrm{Im}}

\def\N{\mathbb N}
\def\Z{\mathbb Z}
\def\m{\mathfrak m}
\def\p{\mathfrak p}

\def\L{\mathcal L} 
\def\R{\mathcal R} 

\def\FF{\mathcal F} 
 
\newtheorem{theorem}{Theorem}[section]
\newtheorem{lemma}[theorem]{Lemma}
\newtheorem{corollary}[theorem]{Corollary}
\newtheorem{proposition}[theorem]{Proposition}
\theoremstyle{definition}
\newtheorem{definition}[theorem]{Definition}

\theoremstyle{remark}
\newtheorem{remark}[theorem]{Remark}
\numberwithin{equation}{section}
\begin{document}
\title[Degenerations of graded Cohen-Macaulay modules]{\textbf{Degenerations of graded Cohen-Macaulay modules}}
\author{Naoya Hiramatsu} 
\address{Department of general education, Kure National College of Technology, 2-2-11, Agaminami, Kure Hiroshima, 737-8506 Japan}
\email{hiramatsu@kure-nct.ac.jp}
\date{\today}
\begin{abstract}

We introduce a notion of degenerations of graded modules. 
In relation to it, we also introduce several partial orders as graded analogies of the hom order, the degeneration order and the extension order. 
We prove that these orders are identical on the graded Cohen-Macaulay modules if a graded ring is of graded finite representation type and representation directed. 
\end{abstract}
\maketitle
\section{Introduction}

The notion of degenerations of modules appears in geometric methods of representation theory of finite dimensional algebras. 
In \cite{Y04} Yoshino gives a scheme-theoretical definition of degenerations, so that it can be considered for modules over a Noetherian algebra which is not necessary finite dimensional. 
Now, a theory of degenerations is considered for derived module categories \cite{JSZ05} or stable module categories \cite{Y11}.  
The degeneration problem of modules has been studied by many authors \cite{B96, R86, Y02, Y04, Z99, Z00}. 
For the study, several order relations for modules, such as the hom order, the degeneration order and the extension order, were introduced, and the connection among them has been studied. 
In the previous paper \cite{HY10}, the author give the complete description of degenerations over a ring of even-dimensional simple hypersurface singularity of type ($A_n$). 

In the present paper we consider degenerations of  graded Cohen-Macaulay modules over a graded Gorenstein ring with a graded isolated singularity. 
First we consider an order relation on a category of graded Cohen-Macaulay modules which is called the hom order (Definition \ref{dfn of hom order}). 
We shall show that it is actually a partial order if a graded ring is Gorenstein with a graded isolated singularity.  
 
In section \ref{Degenerations of graded Cohen-Macaulay modules} we propose a definition of degenerations for graded modules and we state several properties on it. 
We show that if the graded ring is of graded finite representation type and representation directed, then the hom order, the degeneration order and the extension order are identical on the graded Cohen-Macaulay modules. 
We also consider stable analogue of degenerations of graded Cohen-Macaulay modules in section \ref{stable degeneration}.  

\section{Hom order on graded modules}\label{hom order on graded modules}

Throughout the paper $R = \oplus _{i=0}^{\infty} R_i$ be a commutative Noetherian $\N$-graded Cohen-Macaulay ring with $R_0 = k$ a field of characteristic zero.   
A graded ring is said to be ${}^\ast$local if the set of graded proper ideals has a unique maximal element. 
Thus $R$ is ${}^\ast$local since $\m = \oplus _{i > 0} R_i$ is  a unique maximal ideal of $R$. 
We denote by $\grmodR$ the category of finitely generated $\Z$-graded modules whose morphisms are homogenous morphisms that preserve degrees. 
For $i \in \Z$, $M(i) \in \grmodR$ is defined by $M(i)_n = M _{n+i}$. 
Then $\HomR (M, N(i))$ consisting of homogenous morphisms of degree $i$, and we set 
$$
\grHomR (M, N) = \oplus _{i \in \Z} \HomR (M, N(i)).
$$
For a graded prime ideal $\p$ of $R$, we denote by $R_{(\p )}$ a homogenous localization of $R$ by $\p$. 
For $M \in \grmodR$, take a graded free resolution 
$$
\cdots \rightarrow F_l\xrightarrow{d_l}  F_{l-1}  \rightarrow \cdots \rightarrow F_1 \xrightarrow{d_1} F_0 \rightarrow M \rightarrow 0.
$$
We define an $l$th syzygy module $\Omega ^{l} M$ of M by $\Im d_l$. 
We say that a graded $R$-module $M$ is said to be a graded Cohen-Macaulay $R$-module if 
$$
\grExt_{R}^{i} (R/\m, M) = 0 \quad \text{for any } i < d = \dim R.
$$
In particular, this condition is equivalent to 
$$
\grExt_{R}^{i} (M, \omega _R) = 0 \quad \text{for any } i > 0, 
$$
where $\omega _R$ is a ${}^\ast$canonical module of $R$. 
We denote by $\grCMR$ the full subcategory of $\grmodR$ consisting of graded Cohen-Macaulay $R$-modules. 
By our assumption on $R$, $\grmodR$ and $\grCMR$ are Krull-Schmidt, namely each object can be decomposed into indecomposable objects up to isomorphism uniquely.
For $M \in \grmodR$ we denote by $h(M)$ a sequence $(\dim _k M_n) _{n \in \Z}$ of non-negative integers. 
By the definition, it is easy to see that $h(M) = h(N)$ if and only if they have the same Hilbert series. 
Moreover, we also have that $h(M) = h(N)$ if and only if $h(M^{\ast} ) = h(N^{\ast} )$ where $(-)^{\ast} = \grHomR (-, \omega _R)$. 
See \cite[Theorem 4.4.5.]{BH}.

\begin{remark}
If $M$ and $N$ give the same class in the Grothendieck group, {\it i.e.} $[M] = [N]$ as an element of $K_{0} (\grmodR)$ then $h(M) = h(N)$. 
However the converse does not hold in general. 
Let $R = k[x, y]/(x^2-y^2)$ with $\deg x = \deg y = 1$ and set $M = R/(x + y)$ and $N = R /(x-y)$. 
Then $h(M) = h(N)$ and $[M] \not= [N]$ in $K_{0} (\grmodR)$. 
In fact, let $\p$ be an ideal $(x + y)R$. 
Then $\mathrm{rank} _{R_{(\p)}} M_{(\p)} = 1$ and $\mathrm{rank} _{R_{(\p )}} N_{(\p )} = 0$. 
Note that $[M] = [N]$ on $K_{0} (\grmod (R_{(\p)}))$ yields that $[M_{(\p )}] = [N_{(\p )}]$ on $K_{0} (\grmod (R_{(\p)}))$. 
Thus $[M] = [N]$ can never happen.  
\end{remark}

Our motivation of the paper is to investigate the graded degenerations of graded Cohen-Macaulay modules in terms of some order relations. 
For the reason we consider the following relation on $\grCMR$ that is known as the hom order.

\begin{definition}\label{dfn of hom order}
For $M$, $N \in \grCMR$ we define $M \homo N$ if $[M, X] \leq [N, X]$ for each $X \in \grCMR$. 
Here $[M, X] $ is an abbreviation of $\dim_k \HomR (M, X)$.  
\end{definition}

\begin{remark}\label{rmk of hom order}
For $M$, $N \in \grmodR$, $[M, N]$ is finite, and thus we can consider the above relation. 
As a consequence in \cite{B89}, if $R$ is of dimension $0$, $1$ or $2$, $\homo$ is a partial order on $\grCMR$ since $\grCMR$ is closed under kernels in such cases. 
\end{remark}

\begin{lemma}\label{pd finite}
Let $R$ be a graded Gorenstein ring an let $M$ and $N$ be indecomposable graded Cohen-Macaulay $R$-modules.
Suppose that $h(M) = h(N)$. 
For $Y \in \grmodR$ which is of finite projective dimension, we have $[M, Y] = [N, Y]$. 
\end{lemma}

\begin{proof}
For a graded $R$-module $Y$ which is of finite projective dimension, $\Omega ^{i} Y$ is also of finite projective dimension.  
Since $R$ is Gorenstein, we have $\Ext ^{1}_R (M, \Omega ^{i} Y) = 0$ for all graded Cohen-Macaulay $R$-modules $M$. 
Thus, take a graded free resolution of $Y$ and apply $\Hom _R (M, -)$ to the resolution, we get an exact sequence
$$
0 \to \Hom _R (M , F_l  ) \to \Hom _R (M , F_{l-1} ) \to \cdots \to \Hom _R (M , F_0 ) \to \Hom _R (M , Y ) \to 0.
$$
Hence
$$
[M, Y] = \sum_{i = 0}^{l} (-1)^i [M, F_i].
$$ 
We also have 
$$
[N, Y] = \sum_{i = 0}^{l} (-1)^i [N, F_i].
$$
Since $R$ is Gorenstein, $\omega _R = R (l)$ for some $l \in \Z$. 
Then we have an equality $[M, F] = [N, F]$ for each free module.  
Therefore we have
$$
[M, Y] = \sum_{i = 0}^{l} (-1)^i [M, F_i] = \sum_{i = 0}^{l} (-1)^i [N, F_i] = [N, Y].
$$  
\end{proof}

In the paper we use the theory of Auslander-Reiten (abbr. AR) sequences of graded Cohen-Macaulay modules. 
For the detail, we recommend the reader to look at \cite{IT10, AR87, AR88} and \cite[Chapter 15.]{Y}.
We denote by $\grCMPSFR$ the full subcategory of $\grCMR$ consisting of $M \in \grCMR$ such that $M_{(\p )}$ is $R_{(\p )}$-free for any graded prime ideal $\p \not= \m$.

\begin{theorem}\cite{AR87, AR88, Y, IT10}\label{AR-sequence}
Let $(R, \m)$ be a Noetherian $\Z$-graded Gorenstein ${}^{\ast}$local ring. 
Then $\grCMPSFR$ admits AR sequences. 
\end{theorem}

\begin{definition}\label{graded isolated singularity}
We say that $(R, \m)$ is a graded isolated singularity if each graded localization $R_{(\p )}$ is regular for each graded prime ideal $\p$ with $\p \not= \m$.   
\end{definition}

It is easy to see that $\grCMR = \grCMPSFR$ if $R$ is a graded isolated singularity, so that $\grCMR$ admits AR sequences. 
We denote by $\mu (M, Z)$ the multiplicity of $Z$ as a direct summand of $M$.

\begin{theorem}\label{hom order}
Let $R$ be a graded Gorenstein ring with $R_0$ is an algebraically closed field and let $M$ and $N$ graded Cohen-Macaulay $R$-modules. 
Assume that $R$ is a graded isolated singularity. 
Then $[M, X] = [N, X]$ for each $X \in \grCMR$ if and only if $M \cong N$. 
Particularly, $\homo$ is a partial order on $\grCMR$. 
\end{theorem}

\begin{proof}
We decompose $M$ as $M = \oplus M_{i}^{\mu (M, M_i)}$ where $M_i$ are indecomposable graded Cohen-Macaulay $R$-modules. 
If $M_i$ is not free, we can take the AR sequence ending in $M_i$ 
$$
0 \to \tau M_i \to E_i \to M_i \to 0, 
$$
where $\tau M_i$ is an AR translation of $M_i$. 
Apply $\Hom _R (M, -)$ and $\Hom _R (N, -)$ to the sequence, since $k$ is an algebraically closed field, we have
$$
0 \to \Hom _R (M, \tau M_i ) \to \Hom _R (M, E_i )\to \Hom _R (M, M_i ) \to k^{\mu (M, M_i)} \to 0
$$
and
$$
0 \to \Hom _R (N, \tau M_i ) \to \Hom _R (N, E_i )\to \Hom _R (N, M_i ) \to k^{\mu (N, M_i)} \to 0.  
$$
Counting the dimensions of terms, we conclude that $\mu (M, M_i) = \mu (N, M_i)$. 

If $M_i$ is free, we may assume that $M_i = R$. 
Let $\m$ be a ${}^{\ast}$maximal ideal of $R$.  
We consider the Cohen-Macaulay approximation (see \cite{AB89}) of $\m$
$$
0 \to Y \to X \to \m \to 0.
$$  
We note that $X$ is a graded Cohen-Macaulay $R$-module and $Y$ is of finite injective dimension. 
We also note that $Y$ is of finite projective dimension since $R$ is Gorenstein. 
Let $f$ be a composition map of the approximation $X \to \m$ and a natural inclusion $\m \to R$. 
Then we get the following commutative diagram.
$$
\begin{CD} 
0 @>>> K @>>> X @>{f}>> R  @.\\
   @. @AA{g}A   @AA{=}A @AA{\subseteq}A \\
0 @>>> Y @>>> X @>>> \m @>>> 0. \\
\end{CD} 
$$
By a diagram chasing we see that $g$ is surjective, so that $K$ is of finite projective dimension. 
Now we obtain an exact sequence 
$$
0 \to \Hom _R (M, K ) \to \Hom _R (M, X )\to \Hom _R (M, R ) \to k^{\mu (M, R)} \to 0. 
$$
According to Lemma \ref{pd finite}, $[M, K] = [N, K]$.  
Thus
$$
[M, R] + [M, K] - [M, X] = [N, R] + [N, K] - [N, X]. 
$$
Hence $\mu (M, R) = \mu (N, R)$. 
Consequently $M \cong N$. 
\end{proof}

\section{Graded degenerations of graded Cohen-Macaulay modules}\label{Degenerations of graded Cohen-Macaulay modules}

We define a notion of degenerations for graded modules.

\begin{definition}\label{degeneration}
Let $R$ be a Noetherian $\N$-graded ring where $R_0 = k$ is a field and let $V =k[[t]]$ with a trivial gradation and $K$ be the localization by $t$, namely $K = V_t = k(t)$. 
For finitely generated graded $R$-modules $M$ and $N$, we say that $M$ gradually degenerates to $N$ or $N$ is a graded degeneration of $M$ if there is a finitely generated graded $R\otimes _{k} V$-module $Q$ which satisfies the following conditions: 

\begin{itemize}
\item[(1)] $Q$ is flat as a $V$-module.

\item[(2)] $Q \otimes _V V/tV \cong N$ as a graded $R$-module. 

\item[(3)] $Q\otimes _V K \cong M\otimes _{k} K$ as a graded $R\otimes _{k} K$-module.
\end{itemize}
\end{definition}

In \cite{Y04}, Yoshino gives a necessary and sufficient condition for degenerations of (non-graded) modules.  
One can also show its graded version in a similar way. 
See also \cite{R86, Z00}.

\begin{theorem}\label{Zwara sequence}\cite[Theorem 2.2]{Y04}
The following conditions are equivalent for finitely generated graded $R$-modules $M$ and $N$.

\begin{itemize}
\item[(1)] $M$ gradually degenerates to $N$. 

\item[(2)] There is a short exact sequence of finitely generated graded $R$-modules
$$
\begin{CD}
0 @>>> Z @>>> M\oplus Z \ @>>> N @>>> 0.   \\ 
\end{CD}
$$ 
\end{itemize}
\end{theorem}

\begin{remark}\label{remark of extension}
\begin{itemize}
\item[(1)] As Yoshino has shown in \cite{Y04}, the endomorphism of $Z$ in the sequence of Theorem \ref{Zwara sequence} is nilpotent. 
Note that we do not need the nilpotency assumption here. 
Actually, since $\End _R (Z)$ is Artinian, by using Fitting theorem, we can describe the endomorphism as a direct sum of an isomorphism and a nilpotent morphism. 
See also \cite[Remark 2.3.]{Y04}.

\item[(2)] Assume that $M$ and $N$ are graded Cohen-Macaulay modules. 
Then we can show that $Z$ is also a graded Cohen-Macaulay $R$-module. 
See \cite[Remark 4.3.]{Y04}.

\item[(3)]
Assume that there is an exact sequence of finitely generated graded $R$-modules 
$$
\begin{CD}
0 @>>> L @>>> M @>>> N @>>> 0. 
\end{CD}
$$ 
Then $M$ gradually degenerates to $L \oplus N$. 
See \cite[Remark 2.5]{Y04} for instance. 
\item[(4)]
Let  $M$  and $N$  be finitely generated graded $R$-modules and suppose that  $M$  gradually degenerates to $N$. 
Then the modules $M$ and $N$ give the same class in the Grothendieck group. 
\end{itemize}
\end{remark}

We can also prove in a similar way to the proof of \cite[Theroem 2.1.]{Z98} that, for $L$, $M$, $N \in \grmodR$, if $L$ gradually degenerates to $M$ and if $M$ gradually degenerates to $N$ then $L$ gradually degenerates to $N$.  
And one can show that if $L$ gradually degenerates to $M$ then $L \homo M$.  
Consideration into this fact we define partial orders as follows. 
\begin{definition}\label{dfn of orders}
For finitely generated graded $R$-modules $M$, $N$, we define the relation $M \dego N$, which is called the degeneration order, if $M$ gradually degenerates to $N$. 
We also define the relation $M \exto N$ if there are modules $M_i$, $N{'}_i$, $N{''}_i$ and short exact sequences $0 \to N{'}_i \to M_i \to N{''}_i \to 0$ in $\grCMR$ so that $M = M_1$, $M_{i +1} = N{'}_i \oplus N{''}_i$, $1 \leq i \leq s$ and $N = M_{s +1}$ for some $s$.
\end{definition}

For $M \in \grCMR$, we take a first syzygy module of $M^{\ast}$ 
$$
0 \to \Omega^1 M^{\ast} \to F \to M^{\ast} \to 0.
$$
Applying $\grHomR (-, \omega _R)$ to the sequence, we have
$$
0 \to M^{\ast \ast } \cong M \to F^{\ast } \to (\Omega^1 M^{\ast })^{\ast } \to 0.
$$
Then we denote $(\Omega^1 M^{\ast })^{\ast }$ by $\Omega ^{-1} M$.  

	
\begin{lemma}\label{key A}
Let $R$ be a graded isolated singularity with $\omega _R$ and let $M$ and $N$ be graded Cohen-Macaulay $R$-modules. 
Assume that $h(M) = h(N)$ and $M \homo N$. 
Then, for each graded Cohen-Macaulay $R$-module $X$, there exists an  integer $l_X \gg 0$ such that $[M, X(\pm l_X)] = [N, X(\pm l_X)]$.
\end{lemma}

\begin{proof}
For each $X \in \grCMR$, we can take an exact sequence as above
$$
0 \to X \to E \to \Omega ^{-1} X \to 0.
$$
Note that $E$ is a direct sum of $\omega _R (n)$ for some integers $n$.  
Applying $\grHom _R (M, -)$ to the sequence, we have
$$
0 \to \grHom _R (M, X) \to \grHom _R (M, E)\to \grHom _R (M, \Omega^{-1} X)\to \grExt _R (M, X) \to 0. 
$$  
Since $R$ is a graded isolated singularity, $\dim \grExt _R (M, X)$ is finite. 
Thus $\grExt _R (M, X)_{\pm l_1} = 0$ for sufficiently large $l_1 \gg 0$. 
Similarly there also exists an integer $l_2  \gg 0 $ so that $\grExt _R (N, X)_{\pm l_2} = 0$. 

Set $l = max \{ l_1, l_2 \}$. 
Then we have  
$$
0 \to \grHom _R (M, X)_{\pm l} \to \grHom _R (M, E)_{\pm l} \to \grHom _R (M, \Omega^{-1} X)_{\pm l} \to 0. 
$$
Therefore we obtain the equation
$$
[M, X(\pm l)] = [M,  E(\pm l)] -  [M, \Omega ^{-1} X(\pm l)].
$$
We also have  
$$
[N, X(\pm l)] = [N,  E(\pm l)] -  [N, \Omega ^{-1} X(\pm l)].
$$
Suppose that $ [M, X(\pm l)] < [N, X(\pm l)]$. 
Then the following inequality holds.
$$
[M,  E(\pm l)] -  [M, \Omega ^{-1} (X)(\pm l)] <  [N,  E(\pm l)] -  [N, \Omega ^{-1} X(\pm l)]. 
$$  
Since $h(M) = h(N)$, $[M,  E(\pm l)] = [N,  E(\pm l)]$. 
Hence we see that $[M, \Omega ^{-1} X(\pm l)]  > [N, \Omega ^{-1} X(\pm l)] $. 
This is a contradiction since $M \homo N$. 
Therefore we have some integer $l$ such that $ [M, X(\pm l)] = [N, X(\pm l)]$.  
\end{proof}

We say that the category $\grCMR$ is of graded finite representation type if there are only a finite number of isomorphism classes of indecomposable graded Cohen-Macaulay modules up to shift. We note that if $\grCMR$ is of finite representation type, then $R$ is a graded isolated singularity. 
See \cite[Chapter 15.]{Y} for the detail.     

As an immediate consequence of Lemma \ref{key A}, we have the following.

\begin{corollary} \label{key B}
Let $R$ be of finite representation type and let $M$ and $N$ be graded Cohen-Macaulay $R$-modules. 
Assume that $h(M) = h(N)$ and $M \homo N$. Then there are only finitely many  indecomposable graded Cohen-Macaulay $R$-modules $X$ such that $[N, X]-[M, X] >0$. 
\end{corollary}

For graded Cohen-Macaulay $R$-modules $M$ and $N$, we consider the following set of all the isomorphism classes of indecomposable graded Cohen-Macaulay modules 
$$
\FMN = \{ X \ | \ [N, X]-[M, X] >0 \ \} / \cong .
$$
Note from Corollary \ref{key B} that $\FF$ is a finite set if $h(M) = h(N)$ and $M \homo N$. 
We also note that $\omega _R \not \in \FF$ in the case.

\begin{proposition}\cite{R86}\label{Riedtmann}
Let $R$ be a graded Gorenstein ring which is of graded finite representation type and let $M$ and $N$ be graded Cohen-Macaulay $R$-modules. 
Assume that $h(M) = h(N)$ and $M \homo N$. 
Then there exists some graded Cohen-Macaulay $R$-module $L$ such that $M \oplus L$ degenerates to $N \oplus L$. 
\end{proposition}

\begin{proof}
Although a proof of the proposition is given in \cite{R86}, we refer the argument of the proof in the present paper. 
For this reason we briefly recall the proof of the proposition.

Since $R$ is a graded isolated singularity, $\grCMR$ admits AR sequences. 
For each $X \in \FMN$, we can take an AR sequence starting from $X$. 
$$
\Sigma _X : 0 \to X \to E_X \to \tau ^{-1} X \to 0.
$$

Now we consider a sequence which is a direct sum of $[N. X]-[M, X]$ copies of $\Sigma _X$ where $X$ runs through all modules in $\FMN$. 
Namely
$$
\bigoplus _{X \in \FMN} \Sigma _{X}^{[N. X]-[M, X]} =  0 \to U \to V \to W \to 0. 
$$ 
For any indecomposable $Z \in \grCMR$, we obtain 
$$
0 \to \Hom _R(W, Z) \to \Hom _R(V, Z) \to \Hom _R(U, Z) \to k ^{[N, Z]-[M, Z]} \to 0. 
$$
This implies that the equality
$$
[U, Z] + [W, Z] -[V, Z] = [N, Z]-[M, Z], 
$$
thus
$$
[U, Z]+ [W, Z] + [M, Z] = [N, Z]+ [V, Z] 
$$
holds for all $Z \in \grCMR$. 
Hence this yields that 
$$
M \oplus U \oplus W \cong N \oplus V. 
$$
Since $V$ degenerates to $U \oplus W$, therefore $M \oplus V$ degenerates to $M \oplus U \oplus W \cong N \oplus V$. 
\end{proof}

Now we focus on the case that a graded Gorenstein ring is of graded finite representation type and representation directed. 
We say that a graded Cohen-Macaulay ring $R$ is representation directed if the AR quiver of $\grCMPSFR$ has no oriented cyclic paths.  
Bongartz \cite{B96} has studied such a case over finite dimensional $k$-algebras. 
In our graded settings, the similar results hold.  
Actually we shall prove the following.

\begin{theorem}\label{Main theorem}
Let $R$ be a graded Gorenstein ring which is of graded finite representation type and representation directed. 
Then the following conditions are equivalent for $M$ and $N \in \grCMR$. 
 \begin{itemize}
 \item[(1)] $h(M) = h(N)$ and $M \homo N$. 
 \item[(2)] $M \dego N$. 
 \item[(3)] $M \exto N$.  
 \end{itemize}
 \end{theorem}

To prove the theorem, we modify the arguments in \cite{B96}.

\begin{lemma}[Cancellation property]\label{cancellation}
Let $M$, $N$ and $X$ be finitely generated graded $R$-modules. 

\begin{itemize}
\item[(1)] Assume that $[X, M] = [X, N]$. 
If $M \oplus X$ gradually degenerates to $N \oplus X$, $M$ gradually degenerates to $N$. 

\item[(2)] Assume that $R$ is Gorenstein and $M$ and $N$ graded Cohen-Macaulay $R$-modules. 
If $M$ gradually degenerates to $N \oplus F$ for some graded free $R$-module $F$ then $M/F$  gradually degenerates to $N$. 
\end{itemize}
\end{lemma}

\begin{proof}
$(1)$ Since $M \oplus X$ gradually degenerates to $N \oplus X$, there exist an exact sequence 
$$
0 \to W \to  M \oplus X \oplus W \to  N \oplus X \to 0.
$$
We construct a pushout diagram.
$$
\begin{CD} 
@. @. 0 @. 0 @. \\
@. @.   @AAA @AAA \\
@. @. X @= X @. \\
@. @. @AAA @AAA \\
0 @>>> W @>>> M \oplus X \oplus W @>>> N \oplus X @>>> 0 \\
@. @|   @AAA @AAA \\
0 @>>> W @>>> E @>>> N @>>> 0 \\
@. @. @AAA @AAA \\
@. @. 0 @. 0. @. \\
\end{CD} 
$$
For the middle column sequence, 
$$
[X, X] + [X, E] - [X, M \oplus X \oplus W] = [X, E] - [X, M] - [X, W] \geq 0.  
$$
On the other hand, for the bottom row sequence, since $[X, M] = [X, N]$, 
$$
[X, N] + [X, W] - [X, E] = [X, M] + [X, W] - [X, E] \geq 0.  
$$
Thus we have 
$$
[X, X] + [X, E] - [X, M \oplus X \oplus W] = 0.  
$$
This implies that the middle column sequence splits, so that $X \oplus E \cong M \oplus X \oplus W$. 
Therefore $E \cong M \oplus W$ and we get 
$$
 0 \to W \to  M \oplus W \to  N \to 0.
$$
Namely $M$ gradually degenerates to $N$. 

\noindent
$(2)$ Since $M$ gradually degenerates to $N \oplus F$, we have an exact sequence
$$
0 \to Z \to M \oplus Z \to N \oplus F \to 0.
$$
Suppose that $Z$ contains a graded free module $G$ as a direct summand. 
We construct a pushout diagram
$$
\begin{CD} 
@. 0@. 0@.  @. \\
@. @AAA   @AAA @. \\
0 @>>> Z/G @>>> E @>>> N \oplus F @>>> 0 \\
@. @AAA @AAA @| \\
0 @>>> Z @>>> M \oplus Z @>>> N \oplus F @>>> 0 \\
@. @AAA   @AAA @. \\
 @. G @= G @. @. \\
@. @AAA @AAA @. \\
@.0 @. 0. @.  @. \\
\end{CD} 
$$
The left column sequence is a split sequence induced by the decomposition $Z \cong Z/G \oplus G$. 
Note that $E$ is also a graded Cohen-Macaulay module. 
Since $R$ is Gorenstein, the middle column sequence is also split. 
Hence
$$
E \oplus G \cong M \oplus Z \cong M \oplus Z/G \oplus G.
$$
This yields that $E \cong M \oplus Z/G$. 
Hence we may assume that $Z$ has no graded free modules as direct summands. 

Consider a composition of the surjection $M \oplus Z \to N \oplus F$ and the projection $N \oplus F \to F$. 
Then the composition mapping is split, so that $M$ contains $F$ as a direct summand. 
Hence $M \cong M/F \oplus F$ gradually degenerates to $N \oplus F$. 
Since $[F, M/F \oplus F] = [F, M] = [F, N \oplus F]$, by $(1)$, we conclude that $M/F$ gradually degenerates to $N$. 
\end{proof}

For indecomposable graded Cohen-Macaulay modules $M$ and $N$, we write $X \preceq Y$ if $X \cong Y$ or if there exists a finite path from $X$ to $Y$ in the AR quiver of $\grCMPSFR$.

\begin{lemma}\label{minimal}
Let $R$ be a graded Gorenstein ring which is of graded finite representation type and representation directed and let $M$ and $N \in \grCMR$. 
Assume that $h(M) = h(N)$, $M \homo N$ and $M$ and $N$ have no common direct summands. 
Let $X \in \grCMR$ be an indecomposable module such that $\preceq$-minimal with the property $[N, X] - [M, X] > 0$ and $E$ be a middle term of an AR sequence starting from $X$.  
Then $[E, N] = [E,M]$.
\end{lemma}

\begin{proof}
As in the proof of Proposition \ref{Riedtmann}, we can construct the sequence 
$$
0 \to U \to V \to W \to 0
$$
in $\grCMR$ such that $U \oplus W \oplus M \cong V \oplus N$ via taking  a direct sum of AR sequences starting from  modules in $\FMN$. 
This isomorphism implies that $[Z, U] + [Z, W] -[Z, V] = [Z, N] - [Z, M]$ for each $Z \in \grCMR$. 
Thus it is enough to show that the equality $[E, U] + [E, W] -[E, V] = 0$ holds. 
If there is a $Y \in \FMN$ such that, for the AR sequence $0 \to Y \to G \to \tau ^{-1}Y \to 0$, 
$$
[E, Y] + [E, \tau ^{-1}Y] - [E, G] > 0.
$$
If $Y \cong X$, then one can show that $\tau ^{-1}X$ is a direct summand of $E$, so that there is a cyclic path. 
This is a contradiction since $R$ is represented directed. 
Thus $Y$ is not isomorphic to $X$.  
The inequality also show that $\tau ^{-1}Y$ is a direct summand of $E$. 
Thus there is an irreducible map from $X$ to $\tau ^{-1}Y$. 
Hence $X$ is a direct summand of $G$, so that $Y \preceq X$. 
This is a contradiction.  
Consequently, for each $Y \in \FMN$, $[E, Y] + [E, \tau ^{-1}Y] - [E, G] = 0$, therefore $[E, N] = [E, M]$.  
\end{proof}

\begin{proof}[Proof of Theorem \ref{Main theorem}]
The implication (3) $\Rightarrow$ (2) $\Rightarrow$ (1) is trivial. 
Note that $\Ext _R (X, X) = 0$ for an indecomposable $X \in \grCMR$ since $R$ is representation directed. 
Thus the implication (2) $\Rightarrow$ (3) can be shown by the same argument in \cite[3.5.]{Z99}. 
Now we shall show (1) $\Rightarrow$ (2). 
Set $V = \oplus _{W \in \FMN} W$. 
Since $[N, X] - [M, X] = 0$ for any $X \not \in \FMN$, we can show the implication by induction on $d = [N, V] - [M, V]$. 
If $d = 0$, $[N, Z] = [M, Z]$ for each $X \in \grCMR$. 
Thus $M \cong N$. 
Hence assume that $d > 0$ and $M$ and $N$ have no summand in common in the inductive step. 
We  take $X \in \grCMR$ in Lemma \ref{minimal} and let $E$ be a middle term of the AR sequence starting from $X$. 
By virtue of Lemma \ref{minimal} and  Lemma \ref{cancellation} (1) it is enough to show that $E \oplus M \dego E \oplus N$. 
Now we have $E \oplus M \exto X \oplus \tau^{-1} X \oplus M$. 
By the property of AR sequence, for each indecomposable $Y$,
$$
[X, Y] + [\tau ^{-1}X, Y] = [E, Y] + \delta _{X, Y}
$$
where $\delta _{X, Y} = 1$ if $X \cong Y$ and otherwise $0$. 
Thus $X \oplus \tau ^{-1} X \oplus M \homo E \oplus N$. 
We should remark that $\FF _{X \oplus \tau ^{-1} X \oplus M, E \oplus N}$ is contained in $\FMN$ since $X \in \FMN$. 
Then  
$$
\begin{array}{l}
[E \oplus N, V] - [X \oplus \tau^{-1} X \oplus M, V] \\ 
= [N, V] - [M, V] + [E, V] - [X \oplus \tau^{-1} X, V] \\ 
= d -1
\end{array}
$$
By the induction hypothesis, 
$$
E \oplus M \exto X \oplus \tau ^{-1} X \oplus M \dego E \oplus N
$$
so that $E \oplus M \dego E \oplus N$. 
\end{proof}

\begin{remark}
The implication $M \exto N \Rightarrow M \dego N$ does not hold in general. 
Let $R = k[x, y]/(x^2)$ with $\deg x = \deg y = 1$. 
Then $R(-1)$ gradually degenerates to $(x, y^2)R$. 
In fact we have an exact sequence
$$
\begin{CD} 
0 @>>> R/(x) (-2) @>{\tiny \begin{pmatrix}
{x} \\
{0}  \\
\end{pmatrix} }
>> R(-1) \oplus R/(x) (-2) @>{\tiny \begin{pmatrix}
x & y^2
\end{pmatrix} }>> (x, y^2)R @>>> 0. 
\end{CD} 
$$
Since $(x, y^2)R$ is an indecomposable graded Cohen-Macaulay module which is not isomorphic to $R(-1)$, $R(-1) \exto (x, y^2)R$ can never happen.  
See also \cite[Remark 2.5.]{HY10}. 
\end{remark}


\begin{proposition}\label{formula}
Let $R$ be a graded Gorenstein ring which is of graded finite representation type and let $M$ and $N$ be graded Cohen-Macaulay $R$-modules. 
Assume that $h(M) = h(N)$ and $M \homo N$. 
Then for each indecomposable non-free graded Cohen-Macaulay $R$-module $X$ we have the following equality. 
$$
[N, X] - [M, X] = [\tau^{-1}X, N] - [\tau^{-1}X, M].
$$
\end{proposition}

\begin{proof}
Under the assumption, as in the proof of Proposition \ref{Riedtmann}, we can construct an exact sequence 
$$
0 \to U \to V \to W \to 0
$$
in $\grCMR$ such that $[U, X] + [W, X] -[V, X] = [N, X] - [M, X]$. 
It is enough to show that the equality $[U \oplus W, X] -[V, X] = [\tau^{-1}X, U \oplus W] -[\tau^{-1}X, V]$ holds for each indecomposable $X$. 
Moreover the sequence is a direct sum of AR sequences, we may assume that  $0 \to U \to V \to W \to 0$ is an AR sequence starting from $U$.
Let $X \in \grCMR$ be indecomposable and non projective. 
By the property of an AR sequence, we have
$$
[U \oplus W, X] - [V, X] = \delta _{U, X} \quad  [\tau^{-1}X, U \oplus W] -[\tau^{-1}X, V] = \delta _{W, \tau^{-1}X}.
$$  
Since $W \cong \tau^{-1}U$, we can get the equality. 
\end{proof}

\section{Remarks on stable degenerations of graded Cohen-Macaulay modules}\label{stable degeneration}

In the rest of the paper we consider the stable analogue of degenerations of graded Cohen-Macaulay modules. 

Let $R$ be a graded Gorenstein ring where $R_0 = k$ is a field and let $V = k[[t]]$ with a trivial gradation and $K = k(t)$. 
Note that $R\otimes _k V$ and $R\otimes _k K$ are graded Gorenstein rings as well. 
Then $\SgrCM (R\otimes _k V)$ and $\SgrCM (R\otimes _k K)$ are triangulated categories. 
We denote by $\L : \SgrCM (R\otimes _k V) \to \SgrCM (R\otimes _k K)$ (resp. $\R : \SgrCM (R\otimes _k V) \to \SgrCMR$) the triangle functor defined by the localization by $t$ (resp. taking $-\otimes _{V} V/tV$). 
See also \cite[Definition 4.1]{Y11}.

\begin{definition}\label{stably degeneration}
Let $\ulM, \ulN \in \SgrCMR$.
We say that $\ulM$ stably degenerates to $\ulN$ if there exists a graded Cohen-Macaulay module ${\underline{Q}} \in \SgrCM (R\otimes _k V )$ such that $\L({\underline{Q}}) \cong {\underline{M \otimes _k K}}$ in $\SgrCM (R\otimes _k K)$ and $\R({\underline{Q}}) \cong \ulN$ in $\SgrCMR$.
\end{definition}

One can show the following characterization of stable degenerations similarly to the proof of \cite[Theorem 5.1]{Y11}. 

\begin{theorem}\label{conditions for stable degeneration}
Let $R$ be a graded Gorenstein ring where $R_0 = k$ is a field. 
The following conditions are equivalent for graded Cohen-Macaulay $R$-modules $M$ and $N$. 

\begin{itemize}
\item[(1)] $F \oplus M$ degenerates to $N$ for some graded free $R$-module $F$.
\item[(2)] There is a triangle in $\SgrCMR$
$$
\begin{CD}
\ulZ @>>> \ulM \oplus \ulZ \ @>>> \ulN @>>> \ulZ [1] . \\ 
\end{CD}
$$ 
\item[(3)] $\ulM$ stably degenerates to $\ulN$. 
\end{itemize}
\end{theorem}

\begin{proof}
We should note that the implication $(3) \Rightarrow (1)$. 
In our setting, $R\otimes _k V$ and $R\otimes _k K$ are ${}^\ast$local. 
Then graded projective $R\otimes _k V$ (resp. $R\otimes _k K$) -modules are graded free $R\otimes _k V$ (resp. $R\otimes _k K$) -modules. 
Hence we can show the implication as in the Artinian case of the proof of \cite[Theorem 5.1]{Y11}. 
\end{proof}

\begin{remark}
A theory of degenerations for derived categories has been studied in \cite{JSZ05}. 
They have shown that, for complexes $M$, $N$ in the bounded derived category of a finite dimensional algebra, $M$ degenerates to $N$ if and only if there exists a triangle of the form which appears in the above theorem. 
Let $R$ be a graded Gorenstein ring with $R_0 = k$ is an algebraically closed field. 
As shown in \cite{A01, KST07, IT10}, suppose that $R$ has a simple singularity then there exists a Dynkin quiver $Q$ such that we have a triangle equivalence
$$
\SgrCMR \cong \mathrm{D}^{b}(kQ)
$$ 
where $\mathrm{D}^{b}(kQ)$ is a bounded derived category of the category of finitely generated left modules over a path algebra $kQ$. 
By virtue of Theorem \ref{conditions for stable degeneration}, we can describe the degenerations for $\mathrm{D}^{b}(kQ)$ in terms of the graded degenerations for $\grCMR$. 
Since the graded ring $R$ is of graded finite representation type and representation directed, we have already seen them in Theorem \ref{Main theorem}.
\end{remark}

\section*{Acknowledgments}
The author express his deepest gratitude to Tokuji Araya and Yuji Yoshino for valuable discussions and helpful comments.


\end{document}